\documentclass{amsart}

\usepackage{amsmath}
\usepackage{amsfonts}
\usepackage{amssymb,enumerate}
\usepackage{amsthm}
\usepackage[all]{xy}
\usepackage{hyperref}

\newtheorem{lem}{Lemma}[section]
\newtheorem{cor}[lem]{Corollary}
\newtheorem{Fact}[lem]{Fact}
\newtheorem{prop}[lem]{Proposition}
\newtheorem{thm}[lem]{Theorem}

\newtheorem{Defn}[lem]{Definition}
\newtheorem{Notn}[lem]{Notation}
\newtheorem{Ex}[lem]{Example}
\newtheorem{Question}[lem]{Question}
\newtheorem{Property}[lem]{Property}
\newtheorem{Properties}[lem]{Properties}
\newtheorem{Discussion}[lem]{Remark}
\newtheorem{Construction}[lem]{Construction}
\newtheorem{Subprops}{}[lem]
\newtheorem{Para}[lem]{}

\newenvironment{fact}{\begin{Fact}\rm}{\end{Fact}}

\newenvironment{disc}{\begin{Discussion}\rm}{\end{Discussion}}

\newcommand{\ideal}[1]{\mathfrak{#1}}
\newcommand{\m}{\ideal{m}}
\newcommand{\n}{\ideal{n}}

\newcommand{\q}{\ideal{q}}

\newcommand{\coker}{\operatorname{Coker}}

\newcommand{\Ker}{\operatorname{Ker}}

\renewcommand{\geq}{\geqslant}
\renewcommand{\leq}{\leqslant}
\renewcommand{\ker}{\Ker}

\begin{document}

\bibliographystyle{amsplain}

\author{Tirdad Sharif}
\address{Tirdad Sharif, School of Mathematics, Institute for Studies in
Theoretical Physics and Mathematics, P. O. Box 19395-5746, Tehran,
Iran} \email{sharif@ipm.ir}
\urladdr{http://www.ipm.ac.ir/IPM/people/personalinfo.jsp?PeopleCode=IP0400060}
\thanks{The author was supported by a grant from IPM, (No. 83130311).}

\title[A tight closure approach]{A tight closure approach\\to a result of G. Faltings}


\keywords{Complete Intersection Algebras, Deformation Algebras,
Hecke Algebras, Tight Closure} \subjclass[2000]{13A35, 13H10}

\begin{abstract} Using a result of M. Hochster and C. Huneke on $F$-rational rings
a criterion for complete intersection rings of characteristic $p>0$
is presented. As an application, we give a completely different
proof for an algebraic result of G. Faltings that was used by Taylor
and Wiles in \cite{TW} for a simplification of the proof of the
minimal deformation problem.
\end{abstract}

\maketitle

\section{Introduction}
The theory of tight closure was created by Melvin Hochster and Craig
Huneke. An important notion in the theory of tight closure is the
notion of \emph{$F$-rationality}. There is a close connection
between this notion and \emph{rational singularity}, see \cite{sm}.

The main purpose of this note is to use a result of M. Hochster and
C. Huneke in \cite{HH1} on $F$-rational rings, to give a criterion
for complete intersection rings of characteristic $p>0$, see
Theorem~\ref{TS2} for the precise statement.

In \cite{TW} G. Faltings proved an algebraic result that was used by
Taylor and Wiles to reprove the \emph{minimal deformation problem}
by a simpler method. As an application of Theorem~\ref{TS2} we prove
Proposition~\ref{TS4}, which is an extension of Faltings result. Our
approach to prove this proposition is fundamentally different from
Faltings method where he simplifies Taylor and Wiles argument, see
Remark~\ref{TS3}.

In order to give a view for the reader we recall the minimal
deformation problem very briefly.

The modularity conjecture for semistable elliptic curves depends on
a critical conjecture of Wiles that was proved in \cite{TW,W}, see
Conjecture (2.16) of \cite{W}.

Let $\mathcal{D}$ be a deformation theory and let $R_{\mathcal{D}}$
and $T_{\mathcal{D}}$ be the universal deformation and Hecke
algebras associated to $\mathcal{D}$, respectively. The universal
property of $R_{\mathcal{D}}$ implies that there is a homomorphism
$\varphi_{\mathcal{D}}:R_{\mathcal{D}}\rightarrow T_{\mathcal{D}}$.
Wiles conjecture asserts that $\varphi_{\mathcal{D}}$ is an
isomorphism. In the special case, if $\mathcal{D}$ is a minimal
deformation theory, then the above conjecture is called the minimal
deformation problem. For different types of deformation theories and
their relations see Chapter 1 and 2 of \cite{W}.

Wiles method in \cite{W} to prove the above conjecture, was closely
related to show that certain Hecke algebras are complete
intersection that was shown in \cite{TW}.

There is no need for the reader to know the number theoretic
materials of Wiles proof in detail. We refer the interested readers
to \cite{R} for an elegant exposition of the above matters.

\section{Definitions, Notations and The Main Theorem}\label{TS0}
Throughout this note all rings are commutative and Noetherian of
characteristic $p>0$ and all of modules are finite (that is,
finitely generated).

In this note we use the notations $\nu_R(M)$ and $\ell_R(M)$
respectively for the minimal number of generators of $M$ and the
length of $M$ over $R$.

\begin{Defn}\label{C} Let $(R,\m,k)$ be a local ring and let
$R[[X]]=R[[X_1,...,X_n]]$ be the formal power series of $n$
variables over $R$. Write $A=R[[X]]/J$ for an ideal $J$ of $R[[X]]$.
If the local ring $A\otimes_{R}k$ is complete intersection, then $A$
is called a complete intersection $R$-algebra.
\end{Defn}

\begin{Defn} Let $R$ be a ring and let $J=(a_1,...,a_t)$ be an ideal of $R$.
An element $x\in R$ is said to be in the tight closure of $J$ and
write $x\in J^*$ if there is an element $c\in R^0$ such that for all
large $q=p^n$ we have $cx^q\in J^{[q]}$, wherein $R^0$ is the
complement of the union of all minimal primes of $R$ and
$J^{[q]}=(a_1^q,...,a_t^q)$.

An ideal $J$ is called tightly closed if $J=J^*$. The ring $R$ is
called $F$-rational if every parameter ideal of $R$ is tightly
closed and in particular, if every ideal of $R$ is tightly closed,
then it is called $F$-regular.
\end{Defn}

\begin{fact}\label{F} Let $Q$ be a regular local ring, then it is
$F$-regular, see \cite[(4.4)]{HH}.
\end{fact}

To state the main theorem of this note we will make use of the
following simple lemma the proof of which is omitted.

\begin{lem}\label{TS1} Let I be an ideal of R, then $(I^*)^n\subseteq(I^n)^*$ for
$n\geq1$.
\end{lem}

\begin{thm}\label{TS2} Let $(A,\n,k)$ be a local ring of dimension $t$
and let $\q$ be an ideal of $A$, generated by a system of
parameters. Let $\eta\colon A\to Q$ and $\theta\colon Q\to B$ be
local ring homomorphisms such that $Q$ is regular of dimension $t$,
$\theta$ is surjective, and $\eta$ induces an isomorphism between
the residue fields of $A$ and $Q$. Set $\delta =\theta\eta\colon A
\to B$ and $\ell(B/\q^* B)=d$. Assume for an integer
$q=p^n>d^{t}t^{t-1}$ we have $\ell(B/{\q}^{[q]}B)\geq
\ell(A/{\q}^{[q]})$. If the ideal ${\q}^{[q]}$ is tightly closed,
then the canonical homomorphism $\pi:Q/q^*Q\rightarrow B/\q^* B$ is
an isomorphism.

\noindent In particular, if $A$ is equidimensional homomorphic image
of a Cohen-Macaulay ring, then $\pi$ is an isomorphism between
complete intersection rings.
\end{thm}

\begin{proof}
Write $\m$ for the maximal ideal of $Q$ and $J=\Ker\theta.$ We may
assume that $B=Q/J$. From the assumptions it follows that
$\m^d\subseteq J+{\q}^*Q.$ Therefore we have the next inclusions
$\m^{tqd}\subseteq (J+{\q}^*Q)^{tq}\subseteq (J+(\q^*)^{tq}Q)$.
Lemma ~\ref{TS1} implies that $\m^{tqd}\subseteq J +(\q^{tq})^*Q$.
Since $\q$ is generated by $t$ elements it is easy to see that
${\q}^{tq}\subseteq {\q}^{[q]}$. This yields that $\m^{tqd}\subseteq
J+(\q^{[q]})^*Q$ and from our assumptions we have $\m^{tqd}\subseteq
J+\q^{[q]}Q$.

\noindent We claim that $J\subseteq \m^{d+1}$, otherwise choose $u$
in $J$ not in $\m^{d+1}$ and consider the following exact sequence
$$0\longrightarrow \Ker(\alpha)\longrightarrow Q/\m^{c}\stackrel{\alpha}\longrightarrow
Q/\m^{c}\longrightarrow \coker(\alpha)\longrightarrow
0\hspace{1in}(2.5.1)
$$
wherein the map $\alpha$ is the homothety by element $u$ and
$c=tdq$. Obviously, there is a surjective map as the following
$$\coker(\alpha)=Q/(uQ+\m^c)\longrightarrow
Q/(J+{\q}^{[q]}Q)=B/\q^{[q]}B.
$$
Therefore $\ell_Q(\coker(\alpha))\geq \ell_Q(B/{\q}^{[q]}B)$. It is
clear that $B/{\q}^{[q]}B$ has finite length over $Q$. The ideal
$\q$ is an $\n$-primary ideal of $A$, therefore $\ell_{A}
(B/{\q}^{[q]}B)$ is finite and since $k=A/\n\simeq Q/\m$, clearly
the length of $B/{\q}^{[q]}B$ over both of $Q$ and $A$ are equal. By
assumptions $\ell(B/{\q}^{[q]}B)\geq \ell(A/{\q}^{[q]})$, therefore
we have
$\ell_Q(\coker(\alpha))\geq\ell(B/{\q}^{[q]}B)\geq\ell(A/{\q}^{[q]})$
and from \cite[(14.10)]{M} we find that $\ell(A/\q^{[q]})\geq q^t
e(A)\geq q^t$ wherein the symbol $e(A)$ is the multiplicity of $A$.
Thus $\ell_Q(\coker(\alpha))\geq q^t$.

Let $x_1,x_2,...,x_t$ be a minimal generator for $\m.$ The local
ring $Q$ is regular, so the associated graded ring $gr_{\m}(Q)$ is a
polynomial ring in $t$ variables over $k$, see \cite[Page 76,Theorem
9(d)]{se}. This fact implies that $x_j$ for $1\leq j\leq t$ are
analytically independent in $\m$ and so that
$\nu_Q(\m^{c-i})={c-i+t-1\choose t-1}$ for $i\geq 1$. Since $u$ is
not an element of $\m^{d+1}$ therefore $\ker(\alpha)\subseteq
\m^{c-d}/\m^c$ and hence
$\ell_Q(\ker(\alpha))\leq\ell_Q(\m^{c-d}/\m^c)$.

\noindent The following equality is clear
$$\ell_Q(\m^{c-d}/\m^c)=\displaystyle{\sum_{i=1}^{d}}\nu_Q(\m^{c-i}).\hspace{1in}(2.5.2)$$
Now from (2.5.1) we have $\ell_Q(\ker\alpha)=\ell_Q(\coker\alpha)$
and from (2.5.2) we find that $\ell_Q(\m^{c-d}/\m^c)\leq
d{c+t-2\choose t-1}\leq dc^{t-1}=d(tdq)^{t-1}.$ Therefore $q^t\leq
d^{t}t^{t-1}q^{t-1}$ and this contradicts the choice of $q$. Hence
we must have $J\subseteq \m^{d+1}$ and consequently $\m^d\subseteq
J+{\q}^* Q\subseteq \m^{d+1}+{\q}^*Q.$ Thus
$\m^d+{\q}^*Q=\m^{d+1}+{\q}^*Q$ and by using Nakayama's Lemma
$\m^{d+1}\subseteq {\q}^*Q$. This yields that $J\subseteq {\q}^*Q$
and this shows that the canonical homomorphism $\pi : Q/{\q}^*Q
\rightarrow B/{\q}^*B$ is an isomorphism.

Now let $A$ be an equidimensional ring which is homomorphic image of
a Cohen-Macaulay ring. In this case, since ${\q}^{[q]}$ is generated
by a system of parameters and it is tightly closed it follows from
\cite[(4.2.e)]{H} that $A$ is $F$-rational and hence $\q ={\q}^*$.
Now it is easy to see that $\pi$ is an isomorphism between complete
intersection rings.
\end{proof}

\begin{disc}\label{TS3} Let $(\mathcal{O},\m,k)$ be a complete discrete
valuation ring with finite residue field $k$ and let $\mathcal{D}$
be a minimal deformation theory. Let $R$ and $T$ be the universal
deformation and Hecke $\mathcal{O}$-algebras associated to
$\mathcal{D}$, respectively and let $\varphi=\varphi_{\mathcal{D}}$
be the homomorphism described in Section 1. Taylor and Wiles in the
appendix of \cite{TW} for a simplification of some arguments in
Section 3 of \cite{TW} and Chapter 3 of \cite{W} use a commutative
algebraic result of G. Faltings to reprove that $\varphi$ is an
isomorphism between complete intersection algebras. In
Proposition~\ref{TS4} and Corollary~\ref{TS6} we give an extension
of corresponding results due to Taylor-Wiles and Faltings without
the assumption of finiteness of $k$. Our method to prove the above
results is totally different from their methods.

\noindent In the following we refine the definition of a
\emph{(level) n-structure} due to Wiles and Faltings in \cite{TW}.
\end{disc}

\begin{Defn} Let $(\mathcal{O},\m,k)$ be a local ring
 and let $\mathcal{O}[[Y]]=\mathcal{O}[[Y_1,...,Y_t]]$.
 Set $\q=(Y_1,...,Y_t)$ and assume that for an integer $n\geq 1$ we have a
 commutative diagram of local $\mathcal{O}$-algebras as the
 following in which $\varphi$ is surjective and $T$ is a finite free
$\mathcal{O}$-module.
$$
\begin{array}{llll}
 \mathcal{O}[[Y]]\stackrel{\psi_n}\longrightarrow & R_n&\stackrel{\varphi_n}{\longrightarrow} &T_n\\
&\downarrow &&\downarrow \\
& R&\stackrel{\varphi}{\longrightarrow}&T \\
 \end{array}
 $$
By a level $n$-structure we mean the above diagram with the
following properties:

$(1)$ There is a surjective homomorphism
$\lambda_n:\mathcal{O}[[X_1,X_2,...,X_t]]\longrightarrow R_n$.

 $(2)$ The ring homomorphism
 $R_n\longrightarrow T_n$ is surjective.

 $(3)$ From the vertical homomorphisms we get
 the isomorphisms $R_n/{\q}R_n\rightarrow R$ and
 $T_n/{\q}T_n\rightarrow T$.

 $(4)$ The ring $T_n/{\q^{[p^n]}}T_n$ is
 finite and free as a module over $\mathcal{O}[[Y]]/{\q^{[p^n]}}$.

 \vspace{0.05in}\noindent For convenience we represent a level $n$-structure
 by the notation $L_n(\mathcal{O},\varphi,\varphi_n,\psi_n,\lambda_n)$.


\end{Defn}

\begin{prop}\label{TS4} Let $m\geq 1$ be an integer such that for every $n\geq m$,
the level $n$-structure
$L_n(\mathcal{O},\varphi,\varphi_n,\psi_n,\lambda_n)$ has the
following additional property:

\noindent $(5)$ For each $x\in \m$ we have $\psi_n(x)\in \m R_n$.

Then $\varphi$ is an isomorphism between complete intersection
$\mathcal{O}$-algebras.
\end{prop}

\begin{proof} Set
 $\widehat{\square}=\square\otimes_{\mathcal{O}} k$, $\q_n=\q^{[p^n]}$
 and $\mathcal{O}[[X]]=\mathcal{O}[[X_1,X_2,...,X_t]]$. From (1) there
 is a surjective homomorphism $k[[X]]\rightarrow
\widehat{R_n}$. The fifth property implies that
$\psi_n(\m[[Y]])\subseteq \m R_n$. Thus $\psi_n$ induces a local
homomorphism $\widehat{\psi_n}:k[[Y]]\longrightarrow \widehat{R_n}$.
Since $k[[X]]$ is a complete local ring, using \cite[(7.16)]{E} we
can lift the homomorphism $k[[Y]]\rightarrow \widehat{R_n}$ to a
unique homomorphism $k[[Y]]\rightarrow k[[X]]$. Thus from (1), (2)
and (3) we find that there is a surjective homomorphism as the
following
$$\alpha:k[[X]]/{\q}k[[X]]\longrightarrow \widehat{T_n}/{\q}\widehat{T_n}\simeq \widehat{T}.
$$
From (3) it follows that $\widehat{T_n}/{\q}\widehat{T_n}$ is a
finite vector space over $k$. We may assume that the length of
$\widehat{T_n}/{\q}\widehat{T_n}$ over $k[[X]]$ is equal $d$. Now
choose $n\geq m$ such that $p^n>t^{t-1}d^t$. From the fourth
property it is clear that the length of
$\widehat{T_n}/{\q_n}{\widehat{T_n}}$ over $k[[Y]]$ is equal or
greater than the length of $k[[Y]]/{\q_n}k[[Y]]$.

Now by writing $A=k[[Y]]$, $Q=k[[X]]$ and $B=\widehat{T_n}$ from
Theorem~\ref{TS2} and Fact~\ref{F} we get that $\alpha$ is an
isomorphism between complete intersection rings.

On the other hand, $\varphi_n$ is surjective, from this and the
third property we have the following surjective homomorphisms
$$\beta:k[[X]]/{\q}k[[X]]\longrightarrow \widehat{R_n}/{\q}\widehat{R_n}$$
$$\gamma:\widehat{R_n}/{\q}\widehat{R_n}\longrightarrow
\widehat{T_n}/{\q}\widehat{T_n}
$$
such that $\alpha=\gamma\beta$. It was shown that $\alpha$ is an
isomorphism thus from (3) it is clear that $\widehat{\varphi}$ is an
isomorphism between $\widehat{R}$ and $\widehat{T}$ as complete
intersection rings. Since $T$ is a finite free module and $\varphi$
is surjective, Nakayama's Lemma implies that $\varphi$ is an
isomorphism. It is obvious that the local ring $T$ is a homomorphic
image of $\mathcal{O}[[X]]$. Since $\widehat{T}$ is a complete
intersection ring from Definition~\ref{C} it follows that $T$ and so
$R$ are complete intersection $\mathcal{O}$-algebras.
\end{proof}

\begin{disc} In the proof of the above proposition the fifth property enables us
to reduce the level $n$-structure
$L_n(\mathcal{O},\varphi,\varphi_n,\psi_n,\lambda_n)$ to the level
$n$-structure
$L_n(k,\widehat{\varphi},\widehat{\varphi_n},\widehat{\psi_n},\widehat{\lambda_n})$
and then by lifting property which is based on \cite[(7.16)]{E} we
can prove the proposition as a simple corollary of our main theorem.
If we assume that $\mathcal{O}$ is $\m$-adic complete, then
$\mathcal{O}[[X]]$ is $(\m[[X]]+(X))$-adic complete. From (1) we
have a surjective homomorphism $\mathcal{O}[[X]]\longrightarrow
R_n$. Now again by using \cite[(7.16)]{E} we can lift the local
homomorphism $\psi_n$ to a unique local homomorphism
$\mathcal{O}[[Y]]\longrightarrow \mathcal{O}[[X]]$. Now by a similar
method to our proof we can show that $\varphi$ is an isomorphism
between complete intersection $\mathcal{O}$-algebras. In this case,
we don't need to use the fifth property directly. However, in the
following corollary as a simple application of the level
$n$-structures, we make clear the relation between the completeness
of $\mathcal{O}$ and the fifth property. In a sense, we can say that
the assumption of completeness of our base ring implies the fifth
property. We use the next simple lemma the proof of which is
omitted.
\end{disc}

\begin{lem}\label{TS5} Let $L_n(\mathcal{O},\varphi,\varphi_n,\psi_n,\lambda_n)$
be a level $n$-structure, then there is a level $n$-structure
$L'_n(\mathcal{O},\varphi, \varphi'_n, \psi'_n,\lambda'_n)$ such
that $\lambda'_n$ is the canonical epimorphism.
\end{lem}

\begin{cor}\label{TS6} Let $(\mathcal{O},\m,k)$ be an $\m$-adic complete local ring, then
$\varphi$ is an isomorphism between complete intersection
$\mathcal{O}$-algebras.
\end{cor}
\begin{proof} From Lemma~\ref{TS5} it follows that there is a level
$n$-structure such that the homomorphism
$\lambda_n:\mathcal{O}[[X]]\longrightarrow \mathcal{O}[[X]]/J_n=R_n$
is a canonical epimorphism. Therefore for each $g\in
\mathcal{O}[[X]]$ we have $\lambda_n(g)=g+J_n$, where
$J_n=\Ker\lambda_n$. Since $\mathcal{O}$ is $\m$-adic complete hence
$\mathcal{O}[[X]]$ is $(\m[[X]]+(X))$-adic complete. Thus from
\cite[(7.16)]{E} we can lift $\psi_n$ to a unique homomorphism
$\xi:\mathcal{O}[[Y]]\longrightarrow \mathcal{O}[[X]]$ such that
$\xi(x)=x$ for each $x\in \m$. On the other hand, we have
$\lambda_n\xi=\psi_n$. Hence $\psi_n(x)=x+J_n$ and this shows that
$\psi_n(x)\in \m R_n$. Now from Proposition~\ref{TS4} assertion
holds.
\end{proof}

\section*{Acknowledgments}
The author is grateful to Sean Sather-Wagstaff and Irena Swanson for
their useful comments on this note.

\providecommand{\bysame}{\leavevmode\hbox
to3em{\hrulefill}\thinspace}
\providecommand{\MR}{\relax\ifhmode\unskip\space\fi MR }
\providecommand{\MRhref}[2]{%
  \href{http://www.ams.org/mathscinet-getitem?mr=#1}{#2}
} \providecommand{\href}[2]{#2}

\end{document}